\documentclass[12pt, a4paper, english]{amsart}
\usepackage{amsmath} % utilities for mathematics
\usepackage{amsthm} % utilities for theorem environment
\usepackage{amssymb} % loads mathematical fonts
\usepackage{amscd} % utilities for commutative diagrams
\usepackage{array} % core package
\usepackage[mathscr]{eucal} % loads a font accessible by \mathscr{} 
\usepackage[backref=page]{hyperref} % permits navigating on the pdf
\usepackage{caption} %
\usepackage{graphics,graphicx} % permits putting images
\usepackage{xcolor}
\usepackage{tikz,tikz-cd} % tool for diagrams

\usepackage{enumerate} % permits enumerating using letters or other symbols
\usepackage[margin=2.4cm]{geometry}

\usepackage{verbatim}
\usepackage{xcolor}
\makeatletter
\@namedef{subjclassname@2020}{\textup{2020} Mathematics Subject Classification}
\makeatother

\newcommand{\stirling}[2]{\biggl[\genfrac{}{}{0pt}{}{#1}{#2}\biggr]}

\hypersetup{
    colorlinks = true,
    linkbordercolor = {white},
    linkcolor = {magenta},
    anchorcolor = {black},
    citecolor = {blue},
    filecolor = {cyan},
    menucolor = {red},
    runcolor = {cyan},
    urlcolor = {magenta}
}

\usetikzlibrary{automata}

\newtheoremstyle{teoremas}% <name>
{10pt}% <Space above>
{10pt}% <Space below>
{\itshape}% <Body font>
{}% <Indent amount>
{\bfseries}% <Theorem head font>
{}% <Punctuation after theorem head>
{.5em}% <Space after theorem headi>
{}% <Theorem head spec (can be left empty, meaning `normal')>

\theoremstyle{teoremas}
\newtheorem{teo}{Theorem}[section]

\newtheorem{conj}[teo]{Conjecture}

\theoremstyle{definition}
\newtheorem{defi}[teo]{Definition}

\newtheorem{rem}[teo]{Remark}

\DeclareMathOperator{\ehr}{ehr}
\DeclareMathOperator{\Des}{Des}
\DeclareMathOperator{\Vol}{Vol}

\title[]{Ehrhart positivity for a certain class of panhandle matroids}

\author[]{Daniel McGinnis}
\thanks{The author was supported by NSF grant DMS-1839918 (RTG)}

\allowdisplaybreaks

\begin{document}

\maketitle

\begin{abstract}
    We give a combinatorial formula for the Ehrhart coefficients of a certain class of weighted multi-hypersimplices. In a special case, where these polytopes coincide with the base polytope of the panhandle matroid $\textrm{Pan}_{k,n-2,n}$, we show that the Ehrhart coefficients are positive. 
\end{abstract}

\section{Introduction}
The \textit{Ehrhart polynomial} (introduced by Ehrhart in \cite{ehrhart}) of a polytope $P\subset \mathbb{R}^n$ with integral vertices is an invariant of $P$ which counts the number of integer coordinates lying inside integer dilates of $P$. Specifically, the Ehrhart function of $P$, denoted by $\ehr(P,t)$ takes as input a positive integer $t$ and outputs the quantity
\[
\ehr(P,t) = \#\left(tP \cap \mathbb{Z}^n \right),
\]
namely, the number of integer coordinates lying in $tP$. Ehrhart proved that this function is actually a polynomial in $t$ whose degree is the dimension of $P$. Therefore, for $d = \dim(P)$, we may write
\[
\ehr(P,t) = a_dt^d +a_{d-1}t^{d-1} + \cdots + a_0.
\]
An important feature of the Ehrhart polynomial is that $a_d = \Vol(P)$ and $a_{d-1} = \frac{1}{2}\Vol(\partial P)$ (see \cite{beck-robins} \cite{beck-sanyal} for proofs and more information). It is also known that $a_0 =1$, but the remaining coefficients can be negative in general. Thus, an interesting problem which has received a significant amount of attention is to determine families of polytopes having the property that their Ehrhart polynomials have positive coefficients. These polytopes are then called \textit{Ehrhart positive}. Additionally, it is of interest to determine a combinatorial or geometric meaning to the Ehrhart coefficients of Ehrhart positive polytopes. See \cite{liu} for a survey on Ehrhart positivity.

In this paper we extend further upon the work of \cite{FerroniLattice2022}, where it is shown that the Ehrhart polynomials of polytopes of the form
\[
\mathscr{R}_{k,\mathbf{c}}=\left\{x\in [0,c_1]\times \cdots \times [0,c_n] \mid \sum_{i=1}^n x_i=k\right\}.
\]
for positive integers $c_1,\dots,c_n$ and $k$ are Ehrhart positive, and a combinatorial formula is given for the coefficients as well. Note that when $c_1 =\cdots = c_n =1$, we recover the hypersimplex $\Delta_{k,n}$, so the work of \cite{FerroniLattice2022} extends the results of \cite{ferroniPositive}, in which the Ehrhart positivity of hypersimplices is proven using a generating function approach. A combinatorial proof of the the Ehrhart positivity of hypersimplices is given in \cite{hanelyetal} which relies only on an inclusion-exclusion argument. In this paper we attempt to further our current understanding of the Ehrhart coefficients of the hypersimplex by providing a more explicit combinatorial interpretation for these values. Additionally, in the analysis of the Ehrhart polynomial for paving matroids in \cite{hanelyetal}, a class of matroids called \textit{panhandle matroids} were conjectured to be Ehrhart positive. For positive integers $k\leq r < n$, The base polytope for the panhandle matroid $\textrm{Pan}_{k,r,n}$ is given by
\[
\left\{ x \in [0,1]^n \mid \sum_{i=1}^n x_i =k,\, \sum_{i=r+1}^n x_i \leq 1. \right\}
\]
Note the class of panhandle matroid polytopes contain the class of hypersimplices, which can be seen to be the case by taking $r=n-1$.

The class of polytopes described above are examples of  \textit{alcoved polytopes} introduced in \cite{lam-postnikov}, and more specifically, they are contained in the class of polytopes called \textit{weighted multi-hypersimplices} defined in the same paper.

One main result of this paper is to provide a combinatorial description for the Ehrhart coefficients of the weighted multi-hypersimplices of the following form:
\begin{align*}
&
&\left\{(x_1,\dots,x_n) \mid 0\leq x_i\leq c_i \textrm{ for all }1\leq i\leq n-2,\, 0\leq x_{n-1} + x_{n}\leq 1 \textrm{ and } \sum_{i=1}^n x_i=k \right\}
\end{align*}
for positive integers $c_1,\dots,c_{n-2}$ and $k$.

In the case that $\mathbf{c} = (1,\dots,1)$, this polytope coincides with the base polytope associated to the panhandle matroid $\textrm{Pan}_{k,n-2,n}$, and we are able to use our derived combinatorial formula to show that in this case, the polytope is Ehrhart positive. Although a promising approach to proving Ehrhart positivity for the general panhandle matroid $\textrm{Pan}_{k,s,n}$ via a solely enumerative combinatorial conjecture is outlined in \cite{hanelyetal}, our method of proof takes a substantially different route and follows more along the lines with the reasoning of \cite{FerroniLattice2022}. We hope that the ideas presented here will aid future research toward proving Ehrhart positivity for panhandle matroids and other classes of weighted multi-hypersimplices.

We note that the panhandle matroids are certain lattice path matroids \cite{hanelyetal} and hence, they lie within the class of \textit{positroids}, introduced in \cite{PostnikovTotalPositivity}.
It is conjectured in \cite{fjs} that positroids are Ehrhart positive (a matroid is said to be Ehrhart positive if its associated base polytope is Ehrhart positive). 

\begin{conj}[Conjecture 6.3 in \cite{fjs}]\label{con:positroids}
Positroids are Ehrhart positive.
\end{conj}

Since we prove that a certain class of panhandle matroids are Ehrhart positive, our result supports Conjecture \ref{con:positroids}.  \textit{Notched rectangle matroids}, defined in \cite{fan-li} (presented as \textit{cuspidal matroids} in \cite{ferroni2022valuative}), are another class of lattice path matroids which contain panhandle matroids as a subclass. In both \cite{fan-li} and \cite{ferroni2022valuative}, the class of notched rectangle matroids (cuspidal matroids) are conjectured to be Ehrhart positive, hence, the results of this paper supports this conjecture as well.

It was originally  conjectured in \cite{DeLoeraEhrhart2009} that all matroids are Ehrhart positive, moreover, the even stronger conjecture that the larger class of \textit{generalized permutahedra} are Ehrhart positive was stated in \cite{castillo-liu}. However, both of these conjectures are shown to be false in \cite{FerroniMatroids2022} where examples of matroids with negative Ehrhart coefficients with rank between rank 3 and corank 3 are provided. However, it is shown in \cite{fjs} that matroids with rank 2 are Ehrhart positive, and it is noted in the same paper that all matroids of rank 2 are in addition positroids. 

Throughout the progression of ideas in \cite{castillo-liu}, \cite{CastilloTodd2021}, \cite{ferroniPositive}, \cite{fjs}, \cite{hanelyetal} and \cite{FerroniLattice2022}, it became clearer that the Ehrhart positivity of these matroid polytopes requires the introduction of complicated combinatorial structures whose enumeration yields a description of the Ehrhart coefficients, along with a proof of this positivity. We note and emphasize that in \cite{FerroniLattice2022}, such a combinatorial structure is particularly involved; moreover, in \cite{hanelyetal} the conjectured structure is challenging to understand. Our new main contribution is the description of a combinatorial gadget whose enumeration yields an arguably more elegant description of the coefficients of the Ehrhart polynomial of the hypersimplex, and we show how this allows extensions to other weighted multi-hypersimplices. It will become apparent to the reader that the search of a more general structure that covers more (if not, all) weighted multi-hypersimplices would demand a deep and possibly cumbersome combinatorial insight.

It is worth mentioning that the study of the $h^*$-polynomial for polytopes related to those discussed above is an intriguing and active area of research, although we obtain no new results in this direction. For instance, the $h^*$-polynomial of hypersimplices is shown to have very interesting combinatorial properties in \cite{li} and also in \cite{early} \cite{kim} using a different approach. The methods of \cite{kim} are also used in \cite{FerroniLattice2022} to find a combinatorial interpretation for the coefficients of the $h^*$-polynomial of $\mathscr{R}_{k,\mathbf{c}}$. Further research on the $h^*$-polynomial for alcoved polytopes can also be found in \cite{fjs} \cite{sinn-sjoberg} for instance.

\section{The Ehrhart coefficients for hypersimplices revisited}\label{Hypersimplex}

Recall that the hypersimplex $\Delta_{k,n}$ is the polytope given by

\[
\Delta_{k,n}=\left\{x\in [0,1]^n \mid \sum_{i=1}^n x_i=k\right\}.
\]

To write the formula for the Ehrhart polynomial of $\Delta_{k,n}$, we first set up the following notation.
Let $P^n_{a,b} = \sum_{a\leq i_1< \cdots < i_n \leq b} i_1\cdots i_n$, and let ${n\brack m}$ denote the number of permutations of $[n]$ with $m$ cycles, known as the unsigned Stirling numbers of the first kind. Recall that $P_{1,n-1}^{n-m} = {n\brack m}$.

The Ehrhart polynomial for the hypersimplex $\Delta_{k,n}$ is given by
\begin{align}
\ehr(\Delta_{k,n},t)&=\sum_{i=0}^{k-1}(-1)^i\binom{n}{i}\binom{(k-i)t-i+n-1}{n-1}\\
&=\frac{1}{(n-1)!}\sum_{m=0}^{n-1}t^m\sum_{i=0}^{k-1}(-1)^i\binom{n}{i}(k-i)^mP^{n-1-m}_{-i+1,n-1-i}.\label{hypersimplex formula}
\end{align}
This is proven for instance in \cite{ferroniPositive}.

\noindent \textbf{Notation for permutations.} For a permutation $\sigma$, let $C(\sigma)$ be the set of cycles in the cycle decomposition of $\sigma$. Additionally, if we write a permutation $p=[p_1,\dots,p_n]$ in one-line notation, the \textit{descent set of $p$}, $\Des(p)$, is the set of indices $1\leq i\leq n-1$ such that $p_i > p_{i+1}$, and $\textrm{des}(p)$ will denote the cardinality of $\Des(p)$, i.e., $\textrm{des}(p) = |\Des(p)|$. In general, parentheses are used to denote the cycles of a permutation written explicitly in its cycle decomposition, and brackets are used to denote a permutation written in one-line notation\\

\noindent The main combinatorial object of this paper is introduced in Definition \ref{def:main} below.

\begin{defi}\label{def:main}
A \textit{cycle-ordered, weighted permutation} of $[n]$ is a triple $(\sigma,p,w)$ where
\begin{itemize}
    \item $\sigma$ is a permutation of $[n]$.
    \item $p$ is a permutation of $[m]$ satisfying $p(m)=1$ where $m = |C(\sigma)|$ is the number of cycles of $\sigma$.
    \item $w:C(\sigma) \rightarrow \mathbb{N}_0$ is a function.
\end{itemize}

For each cycle $\mathfrak{c}\in C(\sigma)$, we call $w(\mathfrak{c})$ the \textit{weight} of $\mathfrak{c}$. The \textit{total weight} of $(\sigma,p,w)$, which we denote by $w(\sigma)$, is the sum of the weights of each cycle, namely,
\[
w(\sigma) = \sum_{\mathfrak{c}\in C(\sigma)} w(\mathfrak{c}).
\]
If we denote $k= w(\sigma) + \textrm{des}(p)$, then we say $(\sigma,p,w)$ is of \textit{type $(n,m,k)$}.
\end{defi}
We can think of $p$ as ordering the cycles of $\sigma$ according to the smallest elements of the cycles. For example, if $\sigma = (1\ 3)(2\ 6)(4\ 5)$ and $p = [3\ 2\ 1]$, then the order of the cycles of $\sigma$ according to $p$ is $(4\ 5)(2\ 6)(1\ 3)$. Essentially, the order of the cycles according to their smallest elements matches the order of the elements from $p$ when $p$ is written in one-line notation. The reason why we add the condition that $|C(\sigma)|\mapsto 1$ under $p$ is simply because these will be be the only orders of the cycles that will be relevant to us throughout the paper. We also note that because $|C(\sigma)|\mapsto 1$, the cycle of $\sigma$ containing 1 will always come last in the corresponding ordering of its cycles.

Weighted permutations (without an ordering on the cycles) were defined in \cite{hanelyetal}, providing another way to view previously defined combinatorial objects from \cite{ferroniPositive}, which are enumerated by the \textit{weighted Lah numbers}.

Here, we add an ordering to the cycles of weighted permutations to provide a more explicit combinatorial description for the coefficients of the Ehrhart polynomial of certain polytopes, including hypersimplices.

Let $\mathbf{c} = (c_1,\dots,c_n)$ be a tuple of positive integers. A cycle-ordered, weighted permutation $(\sigma,p,w)$ of $[n]$ is said to be \textit{$\mathbf{c}$-compatible} if 
\[
w(\mathfrak{c}) < \sum_{i\in \mathfrak{c}} c_i \textrm{ for all } \mathfrak{c}\in C(\sigma).
\]
We note that the notion of $\mathbf{c}$-compatibility was initially defined in \cite{FerroniLattice2022}.

\begin{teo}\label{thm:CombInt}
The coefficient of $t^m$ in $(n-1)!\ehr(\Delta_{k,n},t)$ is the number of $(1,\dots,1)$-compatible cycle-ordered, weighted permutations $(\sigma,p,w)$ of type $(n,m+1,k)$.
\end{teo}
\begin{proof}
We have that
\begin{align*}
    &(n-1)!\ehr(\Delta_{k,n},t)\\
    &=\sum_{m=0}^{n-1}t^m\sum_{i=0}^{k-1}(-1)^i\binom{n}{i}(k-i)^mP^{n-1-m}_{-i+1,n-1-i}\\
    &=\sum_{m=0}^{n-1}t^m\sum_{i=0}^{k-1}(-1)^i\binom{n}{i}(k-i)^m\sum_{j=0}^{n-1-m}(-1)^jP_{1,i-1}^jP_{1,n-1-i}^{n-1-m-j}\\
    &=\sum_{m=0}^{n-1}t^m\sum_{i=0}^{k-1}\sum_{j=0}^{n-1-m}(-1)^{i-j}\binom{n}{i}(k-i)^m\stirling{i}{i-j}\stirling{n-i}{m+1-i+j}\\
    &=\sum_{m=0}^{n-1}t^m\sum_{i=0}^{k-1}\sum_{j=0}^{n-1-m}\sum_{\substack{A\subset [n]\\ |A|=i}}(-1)^{i-j}(k-i)^m\stirling{i}{i-j}\stirling{n-i}{m+1-i+j}
\end{align*}

We will show that for a fixed set $A\subset [n]$, the quantity $(k-i)^m{i\brack i-j}{n-i\brack m+1-i+j}$ is the number of cycle-ordered, weighted permutations $(\sigma,p, w)$ of type $(n,m+1,k)$ with the following properties. Here, $\ell(\mathfrak{c}) = \sum_{i\in \mathfrak{c}} 1$ denotes the length of $\mathfrak{c}$.
\begin{itemize}
    \item $i-j$ cycles consist only of elements in $A$.
    \item The remaining $m+1-(i-j)$ cycles consist only of elements from $[n]\setminus A$.
    \item For each cycle $\mathfrak{c}$ of $\sigma$ consisting of elements of $A$, $w(\mathfrak{c})\geq \ell(\mathfrak{c}).$
\end{itemize}

First we show this for the case that $i=0$, i.e., we demonstrate that $k^m{n\brack m+1}$ is the number of cycle-ordered, weighted permutations $(\sigma,p,w)$ of type $(n,m+1,k)$. 

Now, ${n\brack m+1}$ is the number of permutations of $[n]$ with $m+1$ cycles, and $k^m$ is the number of functions $f$ that assign each of the cycles that do not contain $1$ a value between $0$ and $k-1$ and assigns the cycle containing 1 the value $k$. Let $v_1 < \dots < v_p$ be the distinct values that were assigned to the cycles. Order the cycles that were assigned $v_1$ in an increasing manner according to their smallest element. Order the cycles that were assigned $v_2$ in the same way, and place them after the cycles that were assigned $v_1$. We continue in this way to obtain an ordering $\mathfrak{c}_1,\dots,\mathfrak{c}_{m+1}$ of the cycles of $\sigma$ (recall that $1\in \mathfrak{c}_{m+1}$). Let $p$ be the permutation of $[m+1]$ corresponding to this ordering. We define the weight $w(\mathfrak{c}_1)$ of $\mathfrak{c}_1$ to be $v_1$, and the weight of $\mathfrak{c_i}$ for $i>1$ is given by
\[
w(\mathfrak{c}_i) = \begin{cases}
f(\mathfrak{c}_i) - f(\mathfrak{c}_{i-1})-1 & \textrm{if } i-1\in \Des(p), \\
f(\mathfrak{c}_i) - f(\mathfrak{c}_{i-1}) & \textrm{otherwise}.
\end{cases}
\]
We have that $w(\sigma)=\sum_{i=1}^{m+1}w(\mathfrak{c}_i)=f(\mathfrak{c}_{m+1})-|\Des(p)|=k-|\Des(p(\sigma))|$. Thus $(\sigma,p,w)$ is of type $(n,m+1,k)$. Furthermore, any cycle-ordered, weighted permutation $(\sigma,p,w)$ of type $(n,m+1,k)$ with its cycles and ordering given by $\mathfrak{c}_1,\dots,\mathfrak{c}_{m+1}$ can be obtained in this way from the function $f$ defined by $f(\mathfrak{c}_i) =  |\Des(p)\cap \{1,\dots,i-1\}| + \sum_{j<i} w(\mathfrak{c}_j)$.

Note that for a given set $A$ of size $i$, ${i\brack i-j}{n-i\brack m+1-i+j}$ is the number of permutations of $[n]$ with $m+1$ cycles where $i-j$ cycles consist only of elements from $A$, and the remaining cycles consist only of elements from $[n]\setminus A$. Thus, by a similar reasoning to the arguments above, we have that $(k-i)^m{i\brack i-j}{n-i\brack m+1-i+j}$ is the number cycle-ordered, weighted permutations of type $(n,m+1,k-i)$, where $i-j$ cycles consist only of elements from $A$, and the remaining cycles consist only of elements from $[n]\setminus A$. Now, for each cycle $\mathfrak{c}$ of $\sigma$ that consists only of elements in $A$, we define $w'(\mathfrak{c})=w(\mathfrak{c}) + \ell(\mathfrak{c})$. Otherwise, we define $w'(\mathfrak{c}) = w(\mathfrak{c})$. The resulting cycle-ordered, weighted permutation $(\sigma,p,w')$ satisfies the bullet points above.

Thus, we have demonstrated that the quantity $(k-i)^m{i\brack i-j}{n-i\brack m+1-i+j}$ is the number of cycle-ordered, weighted permutations that satisfy the above bullet points.

For a cycle-ordered, weighted permutation $(\sigma,p, w)$ where $\sigma = \mathfrak{c}_1 \cdots \mathfrak{c}_{m+1}$, let $I$ be the set of indices $j$ for which $w(\mathfrak{c}_j) \geq \ell(\mathfrak{c}_j)$. For each, $J\subset I$, $(\sigma,p, w)$ contributes $(-1)^{|J|}$ in the sum of the coefficient of $t^m$ above in the term where $A= \bigcup_{j\in J} \mathfrak{c}_j$ and $i=|A|$ (we are slightly abusing notation by associating $\mathfrak{c}_j$ with the set of its elements).

Therefore, the total contribution of $(\sigma,p, w)$ to the sum is $\sum_{J\subset I}(-1)^{|J|}$, which is $0$ if $|I|\geq 1$ and $1$ if $I=\emptyset$. This completes the proof of the theorem.
\end{proof}

Let $A(n,k)$ denote the \textit{Eulerian numbers}, namely, the number of permutations of $[n]$ with $k$ descents, and let $W(n,m+1,\ell)$ denote the number of $(1,\dots,1)$-compatible weighted permutations $(\sigma,w)$ of $[n]$ (here there is no ordering $p$ of the cycles) with $m+1$ cycles. The numbers $W(n,m+1,\ell)$ are precisely the weighted Lah numbers defined in \cite{ferroniPositive}. It is shown in \cite{ferroniPositive} that the coefficient of $t^m$ in $(n-1)!\ehr(\Delta_{k,n},t)$ is given by
\[
\sum_{\ell =0}^{k-1} W(n,m+1,\ell)A(m,k-1-\ell).
\]
This result is proven in part by the use of \textit{Worpitzky's identity} (see \cite{GrahamConcrete} for instance) to further break up and rewrite equation (\ref{hypersimplex formula}). We see that this already implies Theorem \ref{thm:CombInt}. Indeed, $W(n,m+1,\ell)A(m,k-1-\ell)$ is the number of cycle-ordered, weighted permutations $(\sigma,p,w)$ with total weight $\ell$ where the permutation of $\{2,\dots,n\}$ given in one-line notation by $[p_1,\dots,p_{n-1}]$ has $k-1-\ell$ descents. This means that $p=[p_1,\dots,p_n]$ has $k-\ell$ descents since $p_n=1$. Thus, $(\sigma,p,w)$ is of type $(n,m+1,k)$ with total weight $\ell$. Since we are summing up over all $\ell$, we obtain the result of Theorem \ref{thm:CombInt}.

\begin{rem}
We note that the proof of Theorem \ref{thm:CombInt} implicitly contains a proof of Worpitzky's identity, which states that for positive integers $k$ and $m$
\[
k^m = \sum_{i=0}^{m-1} A(m,i)\binom{k+i}{m}.
\]
Indeed, by reasoning similar to that of the proof of Theorem \ref{thm:CombInt}, $k^m$ counts the number of pairs $(p,w)$ consisting of a permutation $p$ of $[m+1]$ where $p(m+1)=1$ and a weight function $w: [m+1]\rightarrow \mathbb{N}_0$ such that $|\Des(p)|+\sum_{i=1}^{m+1}w(i)=k$. On the other hand, $A(m,i)\binom{k+i}{m} = A(m,m-i-1)\binom{k-(m-i)+m}{m}$ is the number of such pairs $(p, w)$ where $|\Des(p)|= m-i$ since $\binom{k-(m-i)+m}{m}$ is the number of ways to distribute $k-(m-i)$ total weight among $m+1$ elements. Since we are summing up over $i$, we obtain Worpitzky's identity.
\end{rem}

In light of the previous remark, the proof of Theorem \ref{thm:CombInt} is similar in spirit to the proof that $[t^m](n-1)!\ehr(\Delta_{k,n},t) = \sum_{\ell =0}^{k-1} W(n,m+1,\ell)A(m,k-1-\ell)$ in \cite{hanelyetal} since they both involve the use of Worpitzky's identity to some extent along with an inclusion-exclusion argument. However, the proof we provide allows us to interpret the coefficient of $t^m$ in equation (\ref{hypersimplex formula}) more explicitly and keeps us from having to needlessly rewrite formulas via Worpitzky's identity later in the paper. Moreover, this interpretation of the coefficients allows for extensions to other weighted multi-hypersimplices as we will show.
\begin{rem}
For a tuple of integers $\mathbf{c}=(c_1,\dots,c_n)$, let $\mathscr{R}_{k,\mathbf{c}}$ be the polytope defined by
\[
\mathscr{R}_{k,\mathbf{c}}=\left\{x\in [0,c_1]\times \cdots \times [0,c_n] \mid \sum_{i=1}^n x_i=k\right\}.
\]
In \cite{FerroniLattice2022}, the authors show that the Ehrhart polynomial $\ehr(\mathscr{R}_{k,\mathbf{c}},t)$ has positive coefficients by expressing the coefficients with a combinatorial formula in terms of $\mathbf{c}$-compatible weighted permutations of $[n]$. Here, we describe the analogue of Theorem \ref{thm:CombInt} for the polytopes $\mathscr{R}_{k,\mathbf{c}}$.

For an integer $v$, let $\rho_{\mathbf{c},i}(v)$ be defined as
\begin{equation}\label{eq:def-rho} \rho_{\mathbf{c},i}(v) := \#\left\{ I\in \binom{[n]}{i} : \sum_{j \in I} c_j = v\right\}.\end{equation}

For example, if $\mathbf{c}$ consists of all $1$'s, then $\rho_{\mathbf{c},i}(i)=\binom{n}{i}$. It is shown in \cite{FerroniLattice2022} that the Ehrhart polynomial of $\mathscr{R}_{k,\mathbf{c}}$ can be written as
\[ \ehr(\mathscr{R}_{k,\mathbf{c}},t)= \sum_{i=0}^{k-1}(-1)^i\sum_{v=0}^{k-1} \binom{t(k-v)+n-1-i}{n-1} \rho_{\mathbf{c},i}(v), \]
and the coefficient of $t^m$ in this polynomial is 
\[
\frac{1}{(n-1)!}\sum_{v=0}^k (k-v)^m\, \sum_{i=0}^{n}(-1)^i\, P^{n-1-m}_{-i+1,n-1-i} \,\rho_{\mathbf{c},i}(v).
\]
Using the same reasoning as in the proof of Theorem \ref{thm:CombInt}, the quantity 
\[
\sum_{v=0}^k (k-v)^m\, \sum_{i=0}^{n}(-1)^i\, P^{n-1-m}_{-i+1,n-1-i} \,\rho_{\mathbf{c},i}(v)
\] 
is the number of $\mathbf{c}$-compatible cycle-ordered, weighted permutations of type $(n,m+1,k)$.

Specifically, following the proof of Theorem \ref{thm:CombInt}, we have that 
\begin{align*}
&\sum_{v=0}^k (k-v)^m\, \sum_{i=0}^{n}(-1)^i\, P^{n-1-m}_{-i+1,n-1-i} \,\rho_{\mathbf{c},i}(v)\\
=& \sum_{v=0}^k\sum_{i=0}^{n}\sum_{j=0}^{n-1-m}\sum_{\substack{A\subset \binom{n}{i}\\ \sum_{u\in A}c_u=v}}(-1)^{i-j}(k-v)^m
\stirling{i}{i-j}\stirling{n-i}{m+1-i+j}.
\end{align*}
And $$(k-v)^m
\stirling{i}{i-j}\stirling{n-i}{m+1-i+j}$$ is the number of cycle-ordered, weighted permutations of type $(n,m+1,k)$ such that for a fixed set $A\in \binom{n}{i}$ where $\sum_{u\in A} c_u =v$ we have the following properties:
\begin{itemize}
    \item $i-j$ cycles consist only of elements in $A$.
    \item The remaining $m+1-(i-j)$ cycles consist only of elements from $[n]\setminus A$.
    \item For each cycle $\mathfrak{c}$ of $\sigma$ consisting of elements of $A$, $w(\mathfrak{c})\geq \sum_{u \in \mathfrak{c}_j} c_u.$
\end{itemize}

For a cycle-ordered, weighted permutation $(\sigma, p, w)$ where $\sigma = \mathfrak{c}_1 \cdots \mathfrak{c}_{m+1}$, let $I$ be the set of indices $j$ such that $w(\mathfrak{c}_j) \geq \sum_{u \in \mathfrak{c}_j} c_u$. Then for each $J\subset I$, we have that $(\sigma, p, w)$ contributes the value $(-1)^{|J|}$ in the term of the sum above where $A = \bigcup_{j\in J} \mathfrak{c}_j$, $v = \sum_{u \in A} c_u$, and $i = |A|$. Hence, the total contribution is $0$ if $|I|\geq 1$ and the contribution is $1$ otherwise, namely, when $(\sigma, p, w)$ is a  $\mathbf{c}$-compatible cycle-ordered, weighted permutations of type $(n,m+1,k)$.
\end{rem}
\section{Formulas for Ehrhart polynomials}

Let $\mathbf{a}=(a_1,\ldots, a_r)$ be an $r$-tuple of positive integers such that $a_1+\cdots+a_r = n$ and let $\mathbf{c}=(c_1,\dots,c_r)$ be an $r$-tuple of positive integers. We denote what we call the weighted multi-hypersimplex of type $(k,\mathbf{a},\mathbf{c})$ as the polytope
    \[ \Delta_{k,\mathbf{a},\mathbf{c}} = \left\{ x\in \mathbb{R}^n_{\geq 0} : \sum_{i=1}^n x_i = k \text{ and } \sum_{j=1+a_1+\cdots+a_{i-1}}^{a_1+\cdots+a_i} x_j \leq c_i\text{ for all $1\leq i\leq r$} \right\},\]
where $a_0 := 0$. We note that these polytopes were explicitly defined in \cite{lam-postnikov}.

Here, we would like to find a formula for the Ehrhart polynomial of $\Delta_{k,\mathbf{a},\mathbf{c}}$ or at least show how it can be computed for particular cases.

For a given integer $u$, the number of nonnegative integer solutions to $\sum_{j=1}^{a_i}x_i=u$ is $\binom{u+a_i-1}{a_i-1}$. We can use this fact to compute 

$$\ehr(\Delta_{k,\mathbf{a},\mathbf{c}},t)=\#\left\{x\in \mathbb{Z}^n_{\geq 0} : \sum_{i=1}^n x_i = kt \text{ and } \sum_{j=1+a_1+\cdots+a_{i-1}}^{a_1+\cdots+a_i} x_j \leq c_it\text{ for all $1\leq i\leq r$}\right\}$$
for any positive integer $t$ as a coefficient of a product of polynomials.

Indeed,
\begin{align*}
\ehr(\Delta_{k,\mathbf{a},\mathbf{c}},t)&=[x^{kt}]\prod_{i=1}^r\left(\sum_{j=0}^{c_it}\binom{j+a_i-1}{a_i-1}x^j\right).
\end{align*}
We can write $\sum_{j=0}^{c_it}\binom{j+a_i-1}{a_i-1}x^j=\frac{1}{(a_i-1)!}D^{a_i-1}(1+x+\cdots+x^{c_it+a_i-1})=\frac{1}{(a_i-1)!}D^{a_i-1}\left(\frac{1-x^{c_it+a_i}}{1-x}\right)$, where $D^k$ denotes the $k$th derivative with respect to $x$. Therefore, we have the following theorem.
\begin{teo}\label{thm:PolyFormula}
For a positive integer $k$ and $r$-tuples $\mathbf{a}$ and $\mathbf{c}$,
\begin{align*}
\ehr(\Delta_{k,\mathbf{a},\mathbf{c}},t)&=[x^{kt}]\prod_{i=1}^r\left(\frac{1}{(a_i-1)!}D^{a_i-1}\left(\frac{1-x^{c_it+a_i}}{1-x}\right)\right).
\end{align*}
\end{teo}
\vskip10pt

\noindent We will now use Theorem \ref{thm:PolyFormula} to explicitly compute the Ehrhart polynomials of the polytopes of the form $\Delta_{k,(1^{(n-2)},2),\mathbf{c}}$ ($1^{(n-2)}$ denotes $n-2$ copies of $1$). Although one could use Theorem \ref{thm:PolyFormula} to explicitly write down the Ehrhart polynomial of $\Delta_{k,\mathbf{a},\mathbf{c}}$ for other fixed choices of $\mathbf{a}$, the formulas become quite cumbersome. Since our combinatorial analysis in what follows does not seem to have a straightforward generalization, we focus only on this particular case. 
Let $\mathbf{c}'=(c_1,\dots,c_{n-2})$ be the first $n-2$ entries of $\mathbf{c}$. We have the following:
\begin{align*}
    &\ehr(\Delta_{k,(1,\dots,1,2),\mathbf{c}},t) = [x^{kt}]\prod_{i=1}^{n-2}\left(\frac{1-x^{c_it+1}}{1-x}\right)\left(\frac{1-(c_{n-1}t+2)x^{c_{n-1}t+1}+(c_{n-1}t+1)x^{c_{n-1}t+2}}{(1-x)^2}\right)\\
    &=[x^{kt}]\frac{1}{(1-x)^n}\Bigg{(}\prod_{i=1}^{n-2}\left(1-x^{c_it+1}\right) -(c_{n-1}t+2)x^{c_{n-1}t+1}\prod_{i=1}^{n-2}\left(1-x^{c_it+1}\right)\\
    &+ (c_{n-1}t+1)x^{c_{n-1}t+2}\prod_{i=1}^{n-2}\left(1-x^{c_it+1}\right)\Bigg{)}\\
    &=\sum_{i=0}^{k-1}(-1)^i\sum_{v=0}^{k-1}\binom{(k-v)t-i+n-1}{n-1}\rho_{\mathbf{c}',i}(v)\\
    &-(c_{n-1}t+2)\sum_{i=0}^{k-1}(-1)^i\sum_{v=0}^{k-c_{n-1}-1}\binom{(k-v-c_{n-1})t-i-1+n-1}{n-1}\rho_{\mathbf{c}',i}(v)\\
    &+ (c_{n-1}t+1)\sum_{i=0}^{k-1}(-1)^i\sum_{v=0}^{k-c_{n-1}-1}\binom{(k-v-c_{n-1})t-i-2+n-1}{n-1}\rho_{\mathbf{c}',i}(v)\\
    &=\frac{1}{(n-1)!}\Bigg{(}\sum_{m=0}^{n-1}t^m\sum_{i=0}^{k-1}(-1)^i\sum_{v=0}^{k-1}(k-v)^mP^{n-1-m}_{-i+1,n-1-i}\rho_{\mathbf{c}',i}(v)\\
    &-c_{n-1}\sum_{m=0}^{n-1}t^{m+1}\sum_{i=0}^{k-1}(-1)^i\sum_{v=0}^{k-c_{n-1}-1}(k-v-c_{n-1})^mP^{n-1-m}_{-i,n-2-i}\rho_{\mathbf{c}',i}(v)\\
    &-2\sum_{m=0}^{n-1}t^m\sum_{i=0}^{k-1}(-1)^i\sum_{v=0}^{k-c_{n-1}-1}(k-v-c_{n-1})^mP^{n-1-m}_{-i,n-2-i}\rho_{\mathbf{c}',i}(v)\\
    &+ c_{n-1}\sum_{m=0}^{n-1}t^{m+1}\sum_{i=0}^{k-1}(-1)^i\sum_{v=0}^{k-c_{n-1}-1}(k-v-c_{n-1})^mP^{n-1-m}_{-i-1,n-3-i}\rho_{\mathbf{c}',i}(v)\\
    &+\sum_{m=0}^{n-1}t^m\sum_{i=0}^{k-1}(-1)^i\sum_{v=0}^{k-c_{n-1}-1}(k-v-c_{n-1})^mP^{n-1-m}_{-i-1,n-3-i}\rho_{\mathbf{c}',i}(v)\Bigg{)}.
\end{align*}

\section{A combinatorial formula for the Ehrhart coefficients}

When $c_{n-1}=1$, we will use the above formula in the proof of Theorem \ref{thm:main} below. First we will need the following definition for the statement of Theorem \ref{thm:main}.

\begin{defi}\label{def:mainComb}
For $\mathbf{a}=(1^{(n-2)},2)$ and $\mathbf{c}=(c_1,\dots,c_{n-2},1)$, we will say that a cycle $\mathfrak{c} \in C(\sigma)$ of a cycle-ordered, weighted permutation $(\sigma,p,w)$ is \textit{properly $(\mathbf{a},\mathbf{c})$-weighted} if the following conditions are satisfied:
\begin{enumerate}
    \item If neither $n-1$ or $n$ are in $\mathfrak{c}$, then $w(\mathfrak{c}) < \sum_{i\in \mathfrak{c}}c_i$.
    \item If $n-1$ or $n$ (or both) are in $\mathfrak{c}$, then $w(\mathfrak{c}) < 1 + \sum_{i\in \mathfrak{c},\, i\leq n-2}c_i$.
\end{enumerate}
If a cycle is not properly $(\mathbf{a},\mathbf{c})$-weighted then we say it is \textit{improperly $(\mathbf{a},\mathbf{c})$-weighted}. Moreover, a cycle-ordered, weighted permutation $(\sigma,p,w)$ is said to be \textit{$(\mathbf{a},\mathbf{c})$-compatible} if each cycle of $\sigma$ is properly $(\mathbf{a},\mathbf{c})$-weighted.
\end{defi}

We make Definition \ref{def:mainComb} solely for the purpose of articulating Theorem \ref{thm:main}. We do not attempt to generalize this definition for other values of $\mathbf{a}$ and $\mathbf{c}$ since we are not confident on what the definition should be in general.

We are now ready to state and prove Theorem \ref{thm:main}.

\begin{teo}\label{thm:main}
The coefficient of $t^m$ in $(n-1)!\ehr(\Delta_{k,(1^{(n-2)},2),(c_1,\dots,c_{n-2},1)},t)$ can be described as the following.

If $k<n-1$ or $m+1>1$, let $S_1$ be the set of $(\mathbf{a},\mathbf{c})$-compatible cycle-ordered, weighted permutations of type $(n,m+1,k)$. Otherwise, if $k=n-1$ and $m+1=1$, let $S_1$ be the set of cycle-ordered, weighted permutations of type $(n,1,n-1)$ (note that these cycle-ordered, weighted permutations are not $(\mathbf{a},\mathbf{c})$-compatible).

Let $S_2$ be the set of cycle-ordered, weighted permutations of type $(n,m+1,k)$ such that
\begin{itemize}
    \item the cycles of $\sigma$ not containing $n-1$ are properly $(\mathbf{a},\mathbf{c})$-weighted,
    \item $n-1$ and $n$ are in different cycles,
    \item the cycle containing $n-1$ is improperly $(\mathbf{a},\mathbf{c})$-weighted,
    \item the cycle $\mathfrak{c}$ containing $n$ satisfies
    \[
    w(\mathfrak{c}) = \sum_{i\in \mathfrak{c},\, i\neq n}c_i.
    \]
\end{itemize}

Let $S_3$ be the set of cycle-ordered, weighted permutations $(\sigma,p,w)$ of type $(n,m,k)$ such that
\begin{itemize}
    \item the cycles not containing $n-1$ or $n$ are properly $(\mathbf{a},\mathbf{c})$-weighted,
    \item $n-1$ and $n$ are in distinct cycles,
    \item the cycle containing $n-1$ is improperly $(\mathbf{a},\mathbf{c})$-weighted,
    \item the cycle $\mathfrak{c}$ containing $n$ satisfies $w(\mathfrak{c}) < \sum_{i\in \mathfrak{c},\, i\neq n}c_i$.
\end{itemize}

If $m>1$, let $S_4$ be the set of cycle-ordered, weighted permutations $(\sigma,p,w)$ of type $(n,m,k)$ such that $n-1$ and $n$ are in the same cycle, which is improperly  $(\mathbf{a},\mathbf{c})$-weighted, and the remaining cycles are properly $(\mathbf{a},\mathbf{c})$-weighted. Otherwise, if $m=1$, let $S_4=\emptyset$.

Then we have that
\[
[t^m](n-1)!\ehr(\Delta_{k,(1^{n-2},2),(c_1,\dots,c_{n-2},1)},t)= |S_1|+|S_2|-|S_3|-|S_4|=|S_1\cup S_2| - |S_3 \cup S_4|.
\]
\end{teo}
\begin{proof}
As we saw before, the Ehrhart polynomial of $(n-1)!\Delta_{k,(1^{(n-2)},2),(c_1,\dots,c_{n-2},1)}$ is equal to
\begin{align*}
    &=\sum_{m=0}^{n-1}t^m\sum_{i=0}^{k-1}(-1)^i\sum_{v=0}^{k-1}(k-v)^mP^{n-1-m}_{-i+1,n-1-i}\rho_{\mathbf{c}',i}(v)\\
    &+\sum_{m=0}^{n-1}t^{m+1}\sum_{i=0}^{k-1}(-1)^{i+1}\sum_{v=0}^{k-2}(k-v-1)^mP^{n-1-m}_{-i,n-2-i}\rho_{\mathbf{c}',i}(v)\\
    &+2\sum_{m=0}^{n-1}t^m\sum_{i=0}^{k-1}(-1)^{i+1}\sum_{v=0}^{k-2}(k-v-1)^mP^{n-1-m}_{-i,n-2-i}\rho_{\mathbf{c}',i}(v)\\
    &+ \sum_{m=0}^{n-1}t^{m+1}\sum_{i=0}^{k-1}(-1)^i\sum_{v=0}^{k-1}(k-v-1)^mP^{n-1-m}_{-i-1,n-3-i}\rho_{\mathbf{c}',i}(v)\\
    &+\sum_{m=0}^{n-1}t^m\sum_{i=0}^{k-1}(-1)^i\sum_{v=0}^{k-2}(k-v-1)^mP^{n-1-m}_{-i-1,n-3-i}\rho_{\mathbf{c}',i}(v),
\end{align*}
where in the second and third line, we absorbed the minus sign in front of the sum into the $(-1)^{i+1}$ term. The coefficient of $t^m$ can be read off to be the sum of the terms
\begin{align*}
    &a_1 = \sum_{i=0}^{k-1}(-1)^i\sum_{v=0}^{k-1}(k-v)^mP^{n-1-m}_{-i+1,n-1-i}\rho_{\mathbf{c}',i}(v)\\
    &a_2 = 2\sum_{i=0}^{k-1}(-1)^{i+1}\sum_{v=0}^{k-2}(k-v-1)^mP^{n-1-m}_{-i,n-2-i}\rho_{\mathbf{c}',i}(v)\\
    &a_3 = \sum_{i=0}^{k-1}(-1)^i\sum_{v=0}^{k-2}(k-v-1)^mP^{n-1-m}_{-i-1,n-3-i}\rho_{\mathbf{c}',i}(v)\\
    &a_4 = \sum_{i=0}^{k-1}(-1)^{i+1}\sum_{v=0}^{k-2}(k-v-1)^{m-1}P^{n-m}_{-i,n-2-i}\rho_{\mathbf{c}',i}(v)\\
    &a_5 = \sum_{i=0}^{k-1}(-1)^i\sum_{v=0}^{k-2}(k-v-1)^{m-1}P^{n-m}_{-i-1,n-3-i}\rho_{\mathbf{c}',i}(v)\\
\end{align*}

Using similar reasoning as in the proof of Theorem \ref{thm:CombInt}, we have the following interpretation of the values above.
\begin{itemize}
    \item If $k<n-1$ or $m+1>1$, let $A_1$ be the set of cycle-ordered, weighted permutations $(\sigma,p,w)$ of type $(n,m+1,k)$ where the cycles that do not contain $n-1$ or $n$ are properly $(\mathbf{a},\mathbf{c})$-weighted. (Note that there are no restrictions on the weight of the cycles containing $n-1$ or $n$.) Otherwise, $k=n-1$, $m+1=1$ and we let $A_1$ be the set of cycle-ordered, weighted permutations of type $(n,1,n-1)$. Then $a_1=|A_1|$.
    
    \item Let $A_2$ be the set of cycle-ordered, weighted permutations $(\sigma,p,w)$ of type $(n,m+1,k)$ where the cycles that do not contain $n-1$ or $n$ are properly $(\mathbf{a},\mathbf{c})$-weighted, $n-1$ and $n$ are in distinct cycles, and exactly one of the cycles containing $n-1$ or $n$ is improperly $(\mathbf{a},\mathbf{c})$-weighted.
    
    Let $A_2'$ be defined in the same way as $A_2$ except both of the cycles containing $n-1$ and $n$ are improperly $(\mathbf{a},\mathbf{c})$-weighted.
    
    Then $a_2=-|A_2|-2|A_2'|$
    
    \item Let $A_3$ be the set of cycle-ordered, weighted permutations $(\sigma,p,w)$ of type $(n,m+1,k)$ such that
    \begin{itemize}
        \item the cycles not containing $n-1$ or $n$ are properly $(\mathbf{a},\mathbf{c})$-weighted,
        \item $n-1$ and $n$ are in distinct cycles, the cycle containing $n-1$ is improperly $(\mathbf{a},\mathbf{c})$-weighted and the weight of the cycle $\mathfrak{c}$ containing $n$ is at least $\sum_{i\in \mathfrak{c}, i\neq n} c_i$.
    \end{itemize}
    
    If $m+1>1$, let $A_3'$ be the set of cycle-ordered, weighted permutations $(\sigma,p,w)$ of type $(n,m+1,k)$ such that
    \begin{itemize}
        \item the cycles not containing $n-1$ or $n$ are properly $(\mathbf{a},\mathbf{c})$-weighted,
        \item $n-1$ and $n$ are in the same cycle, which is improperly $(\mathbf{a},\mathbf{c})$-weighted.
    \end{itemize}
    Otherwise, $m+1=1$ and we let $A_3'=\emptyset$.
    Then $a_3 = |A_3|-|A_3'|$.
    
    \item Let $A_4$ be the set of cycle-ordered, weighted permutations $(\sigma,p,w)$ of type $(n,m,k)$ such that $n-1$ and $n$ are in distinct cycles, the cycle containing $n-1$ is improperly $(\mathbf{a},\mathbf{c})$-weighted, and the cycles not containing $n-1$ or $n$ are properly $(\mathbf{a},\mathbf{c})$-weighted. Then $a_4=-|A_4|$.
    
    \item Let $A_5$ be the set of cycle-ordered, weighted permutations $(\sigma,p,w)$ of type $(n,m,k)$ such that 
    \begin{itemize}
        \item the cycles not containing $n-1$ or $n$ are properly $(\mathbf{a},\mathbf{c})$-weighted,
        \item $n-1$ and $n$ are in distinct cycles, and the cycle containing $n-1$ is improperly $(\mathbf{a},\mathbf{c})$-weighted,
        \item the weight of the cycle $\mathfrak{c}$ containing $n$ is at least $\sum_{i\in \mathfrak{c}, i\neq n} c_i$.
    \end{itemize}
    
    If $m>1$, let $A_5'$ be the set of cycle-ordered, weighted permutations $(\sigma,p,w)$ of type $(n,m,k)$ such that 
    \begin{itemize}
        \item the cycles not containing $n-1$ or $n$ are properly $(\mathbf{a},\mathbf{c})$-weighted.
        \item $n-1$ and $n$ are in the same cycle, which is improperly $(\mathbf{a},\mathbf{c})$-weighted.
    \end{itemize}
    Otherwise, $m=1$ and we let $A_5'=\emptyset$.
    Then $a_5 = |A_5| - |A_5'|$.
\end{itemize}

We can now see that $a_1 + a_2 + a_3 = |S_1|+|S_2|$. To see this, let $(\sigma,p,w)$ be of type $(n,m+1,k)$. 
\begin{itemize}
    \item If $(\sigma,p,w) \in S_1$, then it is contained in $A_1$ and contributes $+1$ to the sum $a_1 + a_2 + a_3$.
    \item If $(\sigma,p,w)\in S_2$, then it is contained in $A_1,\, A_2$, and $A_3$ and contributes $1-1+1=1$ to the sum. Note in the case that $m+1=1$ we have that $S_2=\emptyset$.
    \item If $(\sigma,p,w)$ has some cycle not containing $n-1$ or $n$ that is improperly $(\mathbf{a},\mathbf{c})$-weighted, then it not contained in any $A_i$ or $A_i'$ and does not contribute to the sum.
    \item Assume that $n-1$ and $n$ are in distinct cycles and exactly one of these cycles is improperly $(\mathbf{a},\mathbf{c})$-weighted. Additionally, if $n-1$ is improperly $(\mathbf{a},\mathbf{c})$-weighted, then the cycle $\mathfrak{c}$ containing $n$ has weight less than $\sum_{i\in \mathfrak{c},\, i\neq n}c_i$. If the remaining cycles not containing $n-1$ or $n$ are properly $(\mathbf{a},\mathbf{c})$-weighted, then $(\sigma,p,w)$ is contained in $A_1$ and $A_2$ and thus contributes $1-1=0$ to the sum.
    \item If $n-1$ and $n$ are in distinct cycles which are both improperly $(\mathbf{a},\mathbf{c})$-weighted and the cycles not containing $n-1$ or $n$ are properly $(\mathbf{a},\mathbf{c})$-weighted, then $(\sigma,p,w)$ is contained in $A_1$, $A_2'$, and $A_3$ and contributes $1-2+1=0$ to the sum.
    \item If $n-1$ and $n$ are in the same cycle which is improperly $(\mathbf{a},\mathbf{c})$-weighted, and the remaining cycles are properly $(\mathbf{a},\mathbf{c})$-weighted, then $(\sigma,p,w)$ is contained in $A_1$ and $A_3'$ and contributes $1-1=0$ to the sum.
\end{itemize}

It can be seen, using reasoning similar to the above argument showing $a_1+a_2+a_3 = |S_1|+|S_2|$, that $a_4 + a_5 = -|S_3| - |S_4|$.

This completes the proof.
\end{proof}

Theorem \ref{thm:main} shows that the coefficient of $t^m$ can be viewed as a difference of the cardinality of a set of certain cycle-ordered, weighted permutations with $m+1$ cycles and cycle-ordered, weighted permutations with $m$ cycles. It is possible that a generalization of Theorem \ref{thm:main} to general panhandle matroids would involve an alternating sum of the cardinality of sets of cycle-ordered, weighted permutations with different numbers of cycles. However, a straightforward generalization of Theorem \ref{thm:main} to general panhandle matroids seems unwieldy, and a new understanding of the coefficients in Theorem \ref{thm:main} would likely be needed in order to obtain such a generalization.

When $\mathbf{c}$ consists of all 1's, we are able to use Theorem \ref{thm:main} to show that the coefficients are positive. We note that $\Delta_{k,(1^{(n-2)},2),\mathbf{1}}$ is the base polytope of the panhandle matroid $\textrm{Pan}_{k,n-2,n}$ defined in \cite{hanelyetal}, where a promising method to prove Ehrhart positivity of panhandle matroids is outlined. However, the proof we present uses a different approach than what is suggested there.

\begin{teo}\label{thm:positivity}
Let $\mathbf{1}\in \mathbb{Z}^n$ be the tuple of all $1$'s. Let $0 \leq m \leq n-1$ if $1\leq k\leq n-1$ and $0\leq m\leq 2$ if $k = n-1$. Then the coefficient of $t^m$ in $(n-1)!\ehr(\Delta_{k,(1^{(n-2)},2),\mathbf{1}},t)$ is positive.
\end{teo}
\begin{proof}
Let $S_1,S_2,S_3,S_4$ be defined as in Theorem \ref{thm:main}. We will show that there are injections $f_1: S_3\rightarrow S_1\cup S_2$ and $f_2: S_4\rightarrow S_1\cup S_2$ such that Im$(f_1)\cap \textrm{Im}(f_2) = \emptyset$. Then we will show that there is an element of $S_1\cup S_2$ that is not contained in Im$(f_1)\cup \textrm{Im}(f_2)$. This will complete the proof.

We note that when $k=n-1$, $\Delta_{k,(1^{(n-2)},2),\mathbf{1}}$ is integrally equivalent to the standard 2-dimensional simplex, which is well known to be Ehrhart positive. Therefore, we assume throughout the proof that $1\leq k\leq n-2$; in particular, $\Delta_{k,(1^{(n-2)},2),\mathbf{1}}$ is $(n-1)$-dimensional.

We first note that if $m=0$ or $m=1$, then $S_3,S_4= \emptyset$, so we may assume that $m\geq 2$.

Let $(\sigma,p,w)\in S_3$; we define $(\sigma',p',w') = f_1((\sigma,p,w))$ as follows. We note that the condition on the weight of the cycle containing $n$ implies that this cycle contains some element other than $n$.
First we write the cycle containing $n$ as $(i_1,\dots,i_r,n)$, and we break this cycle into two by $(i_1,\dots,i_{w((i_1,\dots,i_r,n))+1})(i_{w((i_1,\dots,i_r,n))+2},\dots,i_r,n)$. We obtain $\sigma'$ and the ordering of its cycles given by $p'$ by simply replacing $(i_1,\dots,i_r,n)$ with 
\[(i_1,\dots,i_{w((i_1,\dots,i_r,n))+1})(i_{w((i_1,\dots,i_r,n))+2},\dots,i_r,n)
\] 
in the ordering with the following exception: if $(i_1,\dots,i_{w((i_1,\dots,i_r,n))+1})$ contains the element 1, then we place $(i_{w((i_1,\dots,i_r,n))+2},\dots,i_r,n)$ at the beginning in the ordering of the cycles.

For example, if $\sigma$ and the order of its cycles is $(2\ 8)(4\ 3\ 5\ 7)(1\ 6)$ where $w((4\ 3\ 5\ 7))=1$, then $\sigma'$ and the order of its cycles is given by $(2\ 8)(4\ 3)(5\ 7)(1\ 6)$.

We initially define a weight $w^*$ on the cycle $\mathfrak{c}$ of $\sigma'$ containing $n-1$ that will be modified later. Notice that $p'$ has either the same or one more descent than $p$. If $p'$ has the same number of descents as $p$, then we take $w^*(\mathfrak{c}) = w(\mathfrak{c})$, otherwise, we take $w^*(\mathfrak{c}) = w(\mathfrak{c})-1$. The reason why we make this definition is to ensure that the weight function $w'$ defined below will make $(\sigma',p',w')$ of type $(n,m+1,k)$ rather than $(n,m+1,k+1)$.

The weight function $w'$ is defined in the same way as $w$ on all cycles not containing $n-1$ or $n$. 

We define $w'((i_1,\dots,i_{w((i_1,\dots,i_r,n))+1}))$ to be the integer $d$ where $i_{w((i_1,\dots,i_r,n))}$ is the $(d+1)$'th smallest element among $i_1,\dots,i_{w((i_1,\dots,i_r,n))}$. For example, if $i_{w((i_1,\dots,i_r,n))}$ is the smallest among these values, then $d=0$, or if $i_{w((i_1,\dots,i_r,n))}$ is the largest, then $d=w((i_1,\dots,i_r,n))$. 

Let $\mathfrak{c}$ be the cycle containing $n-1$. We define 
\[
w'((i_{w((i_1,\dots,i_r,n))+2},\dots,i_r,n))=\textrm{min}\left\{ \sum_{j=w((i_1,\dots,i_r,n))+2}^r c_{i_j},\, w((i_1,\dots,i_r,n))-d + w^*(\mathfrak{c})  \right\}.
\]
Finally, we define
\[
w'(\mathfrak{c})=w^*(\mathfrak{c}) - (w'((i_{w((i_1,\dots,i_r,n))+2},\dots,i_r,n)) -(w((i_1,\dots,i_r,n))-d)).
\]
The resulting cycle-ordered, weighted permutation $(\sigma',p',w')$ is either properly $(\mathbf{a},\mathbf{1})$-weighted, or the cycle containing $n-1$ is improperly $(\mathbf{a},\mathbf{1})$-weighted, every other cycle is properly $(\mathbf{a},\mathbf{1})$-weighted, and the cycle $(i_{w((i_1,\dots,i_r,n))+2},\dots,i_r,n)$ containing $n$ has weight $\sum_{j=w((i_1,\dots,i_r,n))+2}^r c_{i_j}$. Also, we have that either $w'(\sigma')=w(\sigma)$ and $|\Des(p')|=|\Des(p)|$ or  $w'(\sigma')=w(\sigma)-1$ and $|\Des(p')|=|\Des(p)|+1$, so $(\sigma',p',w')$ is of type $(n,m+1,k)$. Therefore, $(\sigma',p',w')$ is an element of $S_1\cup S_2$.

To see that $f_1$ is injective, we show that the above process in the definition of $f_1$ can be reversed. Let $(\sigma',p',w')$ be in the image of $f_1$. Let $(i_1,\dots,i_{r'})$ be the cycle preceding the cycle containing $n$ (if the cycle containing $n$ is the first cycle, then we mean $(i_1,\dots,i_{r'})$ to be the last cycle), where we have written these elements so that $i_{r'}$ is the $(w'((i_1,\dots,i_{r'}))+1)$'th smallest element among $i_1,\dots,i_{r'}$. Then we merge the cycles $(i_1,\dots,i_{r'})(j_1,\dots,j_s,n)$ into one cycle: $(i_1,\dots,i_{r'},j_1,\dots,j_s,n)$ and we keep the relative ordering of the cycles to obtain the permutation and ordering $\sigma$ and $p$. We then define $w((i_1,\dots,i_{r'},j_1,\dots,j_s,n))=r'-1$, and for the remaining cycles of $\sigma$ not containing $n-1$, $w$ is defined in the same way as $w'$. For the cycle $\mathfrak{c}$ containing $n-1$, we define $w(\mathfrak{c})$ to be 
\[
w(\mathfrak{c}) = w'(\mathfrak{c})+w'((i_1,\dots,i_{r'})) + w'((j_1,\dots,j_s,n)) - (r'-1)
\]
if $p$ has the same number of descents as $p'$, and 
\[
w(\mathfrak{c}) = w'(\mathfrak{c})+w'((i_1,\dots,i_{r'})) + w'((j_1,\dots,j_s,n)) - (r'-1) + 1
\]
if $p$ has one less descent than $p'$.

For example, if $\sigma'$ and the ordering of its cycles is given by $(2\ 9)(4\ 6\ 3\ 7)(5\ 10)(1\ 8)$ with weights $2,2,1,1$ respectively, then we arrange the elements of $(4\ 6\ 3\ 7)$ so that its 3rd smallest element is written on the right: $(3\ 7\ 4\ 6)$. Then we merge the cycles $(3\ 7\ 4\ 6)(5\ 10)$ into one to obtain the permutation $(2\ 9)(3\ 7\ 4\ 6\ 5\ 10)(1\ 8)$ and we give the weights $2,3,1$ respectively.

Thus, $f_1$ is injective.

Now we define the injection $f_2: S_4\rightarrow S_1\cup S_2$. Let $(\sigma,p,w)\in S_4$, and let $(i_1,\dots,i_r,n-1,j_1,\dots,j_s,n)$ be the cycle containing $n-1$ and $n$. We define $f_2((\sigma,p,w))=(\sigma',p',w')$ as follows. The permutation $\sigma'$ and the ordering of its cycles is obtained by simply splitting the cycle $(i_1,\dots,i_r,n-1,j_1,\dots,j_s,n)$ into two cycles $(i_1,\dots,i_r,n-1)(j_1,\dots,j_s,n)$, where in the case that 1 is an element among $i_1,\dots,i_r$, then $(j_1,\dots,j_s,n)$ is placed in the beginning in the ordering of the cycles. We define
\[
w'((j_1,\dots,j_s,n))=\sum_{i=1}^s c_{j_i}
\]
and 
\begin{align*}
&w'((i_1,\dots,i_r,n-1))= \\
&\begin{cases} w((i_1,\dots,i_r,n-1,j_1,\dots,j_s,n)) - \sum_{i=1}^s c_{j_i} & \textrm{if $|\Des(p')|= |\Des(p)|$}\\
w((i_1,\dots,i_r,n-1,j_1,\dots,j_s,n)) -1 - \sum_{i=1}^s c_{j_i} & \textrm{if $|\Des(p')| = |\Des(p)|+1$}.
\end{cases}
\end{align*}

An argument similar to the one above by merging the cycles containing $n-1$ and $n$ shows that this process can be reversed, and thus $f_2$ is injective. Also, the cycle containing $n-1$ is potentially the only cycle that is improperly $(\mathbf{a},\mathbf{1})$-weighted, and since $w'((j_1,\dots,j_s,n))=\sum_{i=1}^s c_{j_i}$ and $(\sigma',p',w')$ is of type $(n,m+1,k)$, $(\sigma', p', w')\in S_1\cup S_2$.

Notice that an element in Im$(f_1)$ has the property that the cycle preceding the cycle containing $n$ does not contain the element $n-1$, and this is not the case for an element in Im$(f_2)$. This shows that Im$(f_1)\cap$Im$(f_2) = \emptyset$.

The above arguments implies that the coefficient is nonnegative. Now we prove that it is in fact positive. To this end, it suffices to show that there is an element in $(S_1\cup S_2)\setminus (\textrm{Im}(f_1)\cup \textrm{Im}(f_2))$. We can assume that $m<n-1$ since we know that the coefficient of $t^{n-1}$ is positive. Let $(\sigma,p,w)$ be a $((1^{(n-1)}),(1^{(n-1)}))$-compatible cycle-ordered, weighted permutation of type $(n-1,m+1,k)$ (recall that the number of such $(\sigma,p,w)$ is the coefficient of $t^m$ in $\ehr(\Delta_{k,n},t)$). We can modify the permutation $\sigma$ by simply replacing the element $n-1$ with $n-1\ n$ in its cycle decomposition. For example, if $\sigma=(1\ 5\ 4)(2\ 3)$, then we obtain the permutation $(1\ 5\ 6\ 4)(2\ 3)$ This results in an $(\mathbf{a},\mathbf{1})$-compatible cycle-ordered, weighted permutation of type $(n,m+1,k)$ where $n-1$ and $n$ are in the same cycle, which is not an element of $(\textrm{Im}(f_1)\cup \textrm{Im}(f_2))$. This completes the proof.
\end{proof}

\begin{rem}
In the proof of Theorem \ref{thm:positivity}, the injection $f_2$ can be defined analogously when $\mathbf{c}$ is of the more general form $(c_1,\dots,c_{n-2},1)$. However, the injection $f_1$ does not appear to carry over in a simple way to this more general case, so the result of Theorem \ref{thm:positivity} may be extended by finding a suitable injection to replace $f_1$.
\end{rem}

\section{Acknowledgements}

The author would like to thank Luis Ferroni and the anonymous referee for helpful comments and suggestions.

\newcommand{\etalchar}[1]{$^{#1}$}
\providecommand{\bysame}{\leavevmode\hbox to3em{\hrulefill}\thinspace}
\providecommand{\MR}{\relax\ifhmode\unskip\space\fi MR }
\providecommand{\MRhref}[2]{%
  \href{http://www.ams.org/mathscinet-getitem?mr=#1}{#2}
}
\providecommand{\href}[2]{#2}

\end{document}